\pdfoutput=1  

\documentclass{amsart}
\usepackage{amssymb}
\usepackage{amsmath}

\usepackage{graphicx}
\usepackage{xcolor}

\usepackage{url}

\newtheorem{theorem}{Theorem}[section]
\newtheorem{lemma}[theorem]{Lemma}

\theoremstyle{definition}

\theoremstyle{remark}

\numberwithin{equation}{section}

\newcommand{\Z}{\ensuremath{\mathbb{Z}}}

\definecolor{amaranth}{rgb}{0.9, 0.17, 0.31} 
\definecolor{carrotorange}{rgb}{0.93, 0.57, 0.13} 
\definecolor{citrine}{rgb}{0.89, 0.82, 0.04} 
\definecolor{dartmouthgreen}{rgb}{0.05, 0.5, 0.06} 
\definecolor{ballblue}{rgb}{0.13, 0.67, 0.8} 
\definecolor{ceruleanblue}{rgb}{0.16, 0.32, 0.75} 
\definecolor{amethyst}{rgb}{0.6, 0.4, 0.8} 
\definecolor{amber}{rgb}{1.0, 0.75, 0.0} 
\definecolor{burlywood}{rgb}{0.87, 0.72, 0.53} 

\begin{document}

\title[Asymmetric L-space knots of braid index four]{On asymmetric hyperbolic L--space knots of braid index four}


\author[K. Baker]{Kenneth L. Baker}
\address{Department of Mathematics, University of Miami, Coral Gables, FL 33146, USA }
\email{k.baker1@miami.edu}

\author[M. Teragaito]{Masakazu Teragaito}
\address{Department of Mathematics and Mathematics Education, Hiroshima University,
1-1-1 Kagamiyama, Higashi-hiroshima 7398524, Japan.}
\email{teragai@hiroshima-u.ac.jp}
\thanks{
The first author has been partially supported by the Simons Foundation gift \#962034.
The  second author has been partially supported by JSPS KAKENHI Grant Number JP25K07004.}


\date{\today}



\begin{abstract}
A knot is called an L--space knot if it admits a positive Dehn surgery yielding
an L--space.
In the SnapPy census, there are exactly 9 asymmetric L--space knots.
Among them, the knot $t12533$ is the only known example of braid index $4$.
We generalize this knot, and give the first infinite family
of asymmetric hyperbolic L--space knots of braid index $4$.
\end{abstract}

\maketitle


\section{Introduction}\label{sec:intro}

A knot is called an \textit{L--space knot\/} if it admits a positive Dehn surgery yielding an L--space.
If the symmetry group of the knot complement is trivial, then
a knot is said to be \textit{asymmetric}.
All the initially known examples of L--space knots were strongly invertible, but
Baker and Luecke \cite{BL} discovered infinitely many asymmetric hyperbolic L--space knots.
Later, it is verified that there are exactly 9 asymmetric L--space knots among 632 L--space knots in Dunfield's list \cite{ABG,BK,D1,D2}.
Among them, the knot $t12533$ is the only one of braid index $4$.
(For L--space knots, it is known that if its braid index is less than $4$, then it is strongly invertible.  See \cite[Section 1]{ABG}.)
Hence it is asked in \cite[Question 9]{ABG} whether there exist other asymmetric L--space knots of braid index $4$.

The purpose of this paper is to give a positive answer to this question.

\begin{theorem}\label{thm:main}
There are infinitely many asymmetric hyperbolic L--space knots of braid index $4$.
\end{theorem}


We will exhibit an infinite family of such knots.
For each integer $n\ge 1$, consider a positive $4$--braid:
\begin{equation*}\label{eq:braid2}
\beta_n=[(1,2,3)^4,2,1,3,2,2,1,1,2,1,1,1,1,1,2^{2n}],
\end{equation*}
where an integer $i$ denotes the standard generator $\sigma_i$ in the braid group $B_4$ of $4$ strands.
Let  $K_n$ be the closure of $\beta_n$.
Note that $(1,2,3)^{4}$ corresponds to a positive full twist on the $4$--strands.
In particular, $K_1$ is the knot $t12533$.
Since $K_n$ is the closure of a positive braid, it is fibered (\cite{S}), and 
so one may calculate that
$K_n$ has genus $n+11$.
Hence $K_m\ne K_n$ if $m\ne n$.

Thus Theorem \ref{thm:main} immediately follows from
the following.

\begin{theorem}\label{thm:main2}
Let $K_n$ be the knot defined as above.
Then $K_n$ is an asymmetric, hyperbolic L--space knot of braid index $4$ for any $n\ge 1$.
\end{theorem}


\section{On twist families of L--space knots}

We first prepare a general lemma to show that each member of a twist family is an L--space knot.

\begin{lemma}\label{lem:twist-family}
Let $K \cup c$ be a link of an L--space knot $K$ and an unknot $c$ with linking number $w = \ell k(K,c)>1$.
Let $K_n$ be the image of $K$ after $(-1/n)$--surgery on $c$ for integers $n \in \Z$.
If there exists a slope $r \geq 2g(K)-1$ such that $(K \cup c)(r,0)$ is also an L--space, then $K_n$ is an L--space knot for all integers $n  \geq 0$.
\end{lemma}

\begin{proof}
Since $K$ is an L--space knot and $r \geq 2g(K)-1$, the surgered manifold $K(r)$ is an L--space by \cite{OS3}.
Let $c_r$ be the image of $c$ in $K(r)$ and continue to use the standard meridian and longitude of $c$ to parametrize slopes for $c_r$.   By assumption $K(r) = (K \cup c)(r, \infty)$ and $(K \cup c)(r,0)$ are both L--spaces, so both $0$ and $\infty$ are L--space slopes for $c_r$.  By \cite{RR}, the set of L--space slopes for a knot in a closed $3$-manifold is either empty or an interval disjoint from the homological longitude.  A homological computation (akin to \cite[Lemma 3.3]{G}) shows that if $r=p/q$ for coprime integers $p$ and $q$, then the homological longitude of $c_r$ has slope $w^2 \cdot q/p$.  In particular, since $r$ is positive and $w$ is non-zero, the slope of the homological longitude of $c_r$ is also positive.  Hence all negative slopes between the slopes $0$ and $\infty$ are L--space slopes for $c_r$.  In particular, the slopes $-1/n$ for integers $n \geq 0$ are L--space slopes for $c_r$.  Hence  for each integer $n\geq 0$, the filling $(K \cup c)(r,-1/n) = K_n(r+w^2n)$ is an L-space so that the knot $K_n$ is an L--space knot.
\end{proof}

Let us introduce a link $K_0\cup c$ as illustrated in Figure \ref{fig:2ndlink}.
Then $K_n$ is the image of $K_0$ after $(-1/n)$--surgery on $c$.
We note that $K_0$ is a strongly invertible L--space knot $m239$ with genus $11$  in Dunfield's census.

\begin{figure}[h]
\includegraphics*[bb=0 0 393 102, width=10cm]{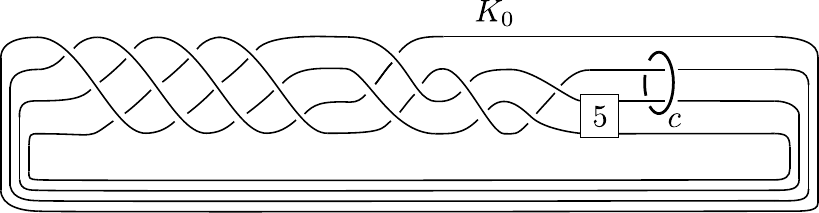}
\caption{The link $K_0\cup c$.
A box with integer $5$ contains right handed $5$ half twists.}
 \label{fig:2ndlink}
\end{figure}

\begin{lemma}\label{lem:lens2}
For the link $K_0\cup c$, $(K_0\cup c)(29,0)$ is a lens space.
\end{lemma}

\begin{proof}
We start from a $4$--component link $k\cup c\cup c_1\cup c_2$ as shown in Figure \ref{fig:2ndsurgery1}
(Top Left).
 Then the surgery diagram $(k\cup c\cup c_1\cup c_2)(-7,0,-1/2,-1/2)$ represents
 $(K_0\cup c)(29,0)$ as shown there.
 
\begin{figure}[h]
\includegraphics*[bb=0 0 629 536, width=12cm]{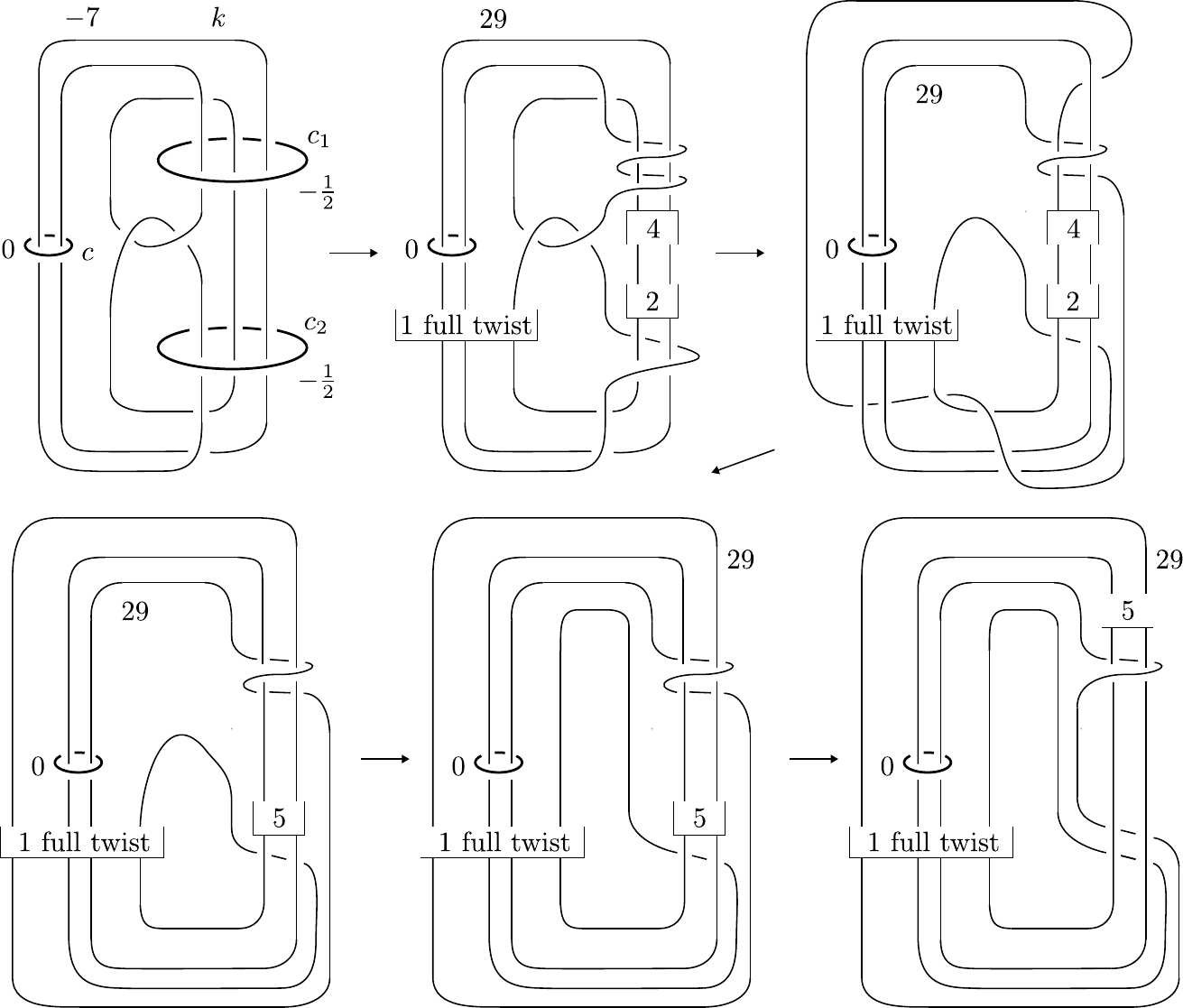}
\caption{The surgery diagram $(k\cup c\cup c_1\cup c_2)(-7,0,-\frac{1}{2},-\frac{1}{2})$
represents $(K_0\cup c)(29,0)$.
A box with integer $i$ contains vertical right handed $i$ half twists.}
 \label{fig:2ndsurgery1}
\end{figure}

On the other hand, Kirby--Rolfsen calculus shows that
the surgery diagram represents a lens space $L(4,-1)$.
See Figure \ref{fig:2ndsurgery2}.
\end{proof}

\begin{figure}[h]
\includegraphics*[bb=0 0 661 600, width=12cm]{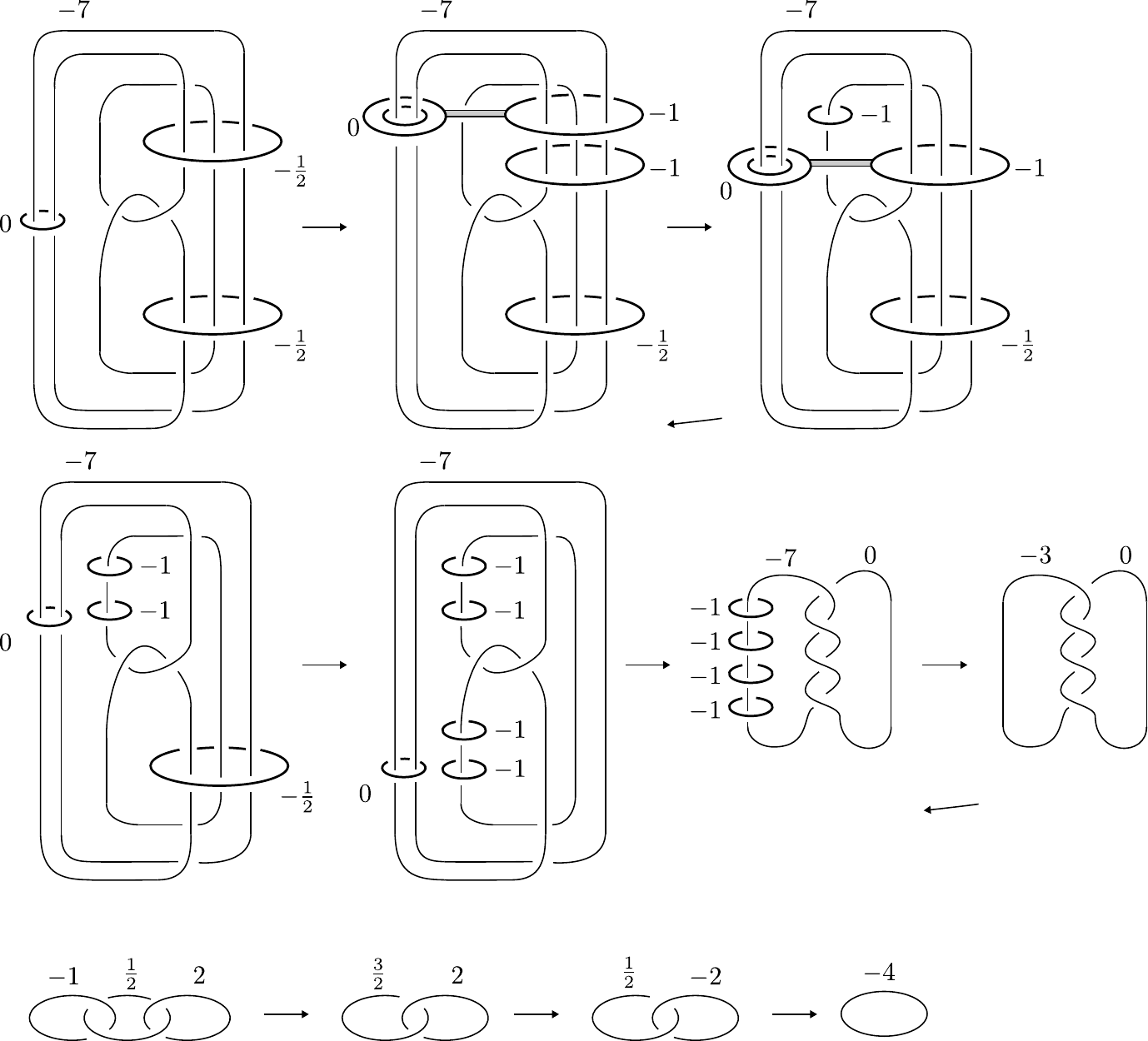}
\caption{A series of handle slides and moves.}
 \label{fig:2ndsurgery2}
\end{figure}

\begin{lemma}\label{lem:lspace}
For each $n\ge 1$, $K_n$ is an L--space knot.
\end{lemma}

\begin{proof}
This immediately follows from Lemmas \ref{lem:twist-family} and \ref{lem:lens2}.
\end{proof}


\section{Proof of Theorem \ref{thm:main2}}

\begin{lemma}\label{lem:braid}
The braid index of $K_n$ is $4$.
\end{lemma}

\begin{proof}
By the definition of $K_n$, it is the closure of a positive $4$--braid which contains one full twist on the $4$ strands.
Such a knot is said to be twist positive in \cite{KM}.
Then  its braid index $4$ by \cite{FW,Mo}.
\end{proof}

\begin{lemma}\label{lem:asymmetry}
The knot $K_n$ is hyperbolic and asymmetric.
\end{lemma}

\begin{proof}
Recall the link $L=K_0\cup c$ as shown in Figure \ref{fig:2ndlink}.
Performing $(-1/n)$--surgery on $c$ changes  $K_0$ into $K_n$.
SnapPy tells us that the complement  $N$ of $L$ is hyperbolic and asymmetric.
Also, the shortest geodesic has length at least $1.48$.
Hence, following \cite[Corollary 2.5]{BKMc} which uses the reformulation of \cite[Theorem 7.28]{FPS} given in \cite[Theorem 2.4]{BKMc}, any manifold obtained by filling the cusp of $c$ of $N$ with slopes of normalized length at least $10.1$ will also be hyperbolic and asymmetric. (Though \cite[\S2.4]{BKMc} is stated for one-cusped manifolds, it extends to our situation here.  Indeed, \cite[Theorem 7.28]{FPS} applies to fillings of subsets of cusps of manifolds with multiple cusps.)

From the diagram of $L=K_0\cup c$,
SnapPy gives the cusp shape of $N$ corresponding to $c$
as $z=0.05249786712 + 0.61334493863i$
using the standard meridian-longitude basis.  That is, the parallelogram in $\mathbb{C}$ with vertices $0$, $1$, $z$, and $z+1$ represents the similarity class of the cusp where the meridian corresponds to the edge from $0$ to $1$ and the longitude corresponds to the edge from $0$ to $z$. 

Thus, with this shape, 
the slope $-\frac{1}{n}$ has  length $|nz-1|$ and normalized length
$L(-\frac{1}{n})=\frac{ |nz-1| } { \sqrt{ |\mathrm{Im}(z)| } }$.
Then $L(-\frac{1}{n})\ge 10.1$ only if $n\ge 13$.
A direct check by SnapPy confirms that $K_n$ is also hyperbolic and asymmetric for integers $0<n\le 12$.
(As stated before, $K_0$ is a strongly invertible, hyperbolic L--space knot $m239$.)
\end{proof}

\begin{proof}[Proof of Theorem \ref{thm:main2}]
This follows from Lemmas \ref{lem:lspace}, \ref{lem:braid} and  \ref{lem:asymmetry}.
\end{proof}



\bibliographystyle{amsplain}

\end{document}